\documentclass[11pt]{article}
\usepackage{amsmath,amsfonts,amsthm,amssymb,ifthen}

\newtheorem{theorem}{Theorem}

\newtheorem{lemma}[theorem]{Lemma}

\newcommand{\tr}[1]{\mathrm{Tr}\left[#1 \right]}

\newcommand{\bb}[1]{\mathbb{#1}} 
\newcommand{\dPhi}{\Delta_b}

 \newcommand{\R}{\bb{R}}

\title{An Elementary Proof of the Restricted Invertibility Theorem%
  \thanks{ This material is based upon work supported by the National
    Science Foundation under grants CCF-0634904, CCF-0634957 and
    CCF-0915487.  Any opinions, findings, and conclusions or
    recommendations expressed in this material are those of the
    author(s) and do not necessarily reflect the views of the National
    Science Foundation.  }} \author{
  Daniel A. Spielman \\
  Department of Computer Science\\
  Program in Applied Mathematics\\
  Yale University \and
  Nikhil Srivastava\\
  Department of Computer Science\\
  Yale University} \date{\today}
\begin{document}
\maketitle
\begin{abstract}
We give an elementary proof of a generalization of Bourgain and Tzafriri's 
  Restricted Invertibility Theorem, which says roughly that any matrix
  with columns of unit length and bounded operator norm has a large coordinate
  subspace on which it is well-invertible.
Our proof gives the tightest known form of this result, 
  is constructive,
  and provides a deterministic polynomial time algorithm for finding the 
  desired subspace. 
\end{abstract}
\section{Introduction}
In this note we study the following well-known theorem of Bourgain and Tzafriri.
\begin{theorem}[Restricted Invertibility \cite{bt87}]\label{thm:bt}
There are universal constants $c,d>0$, such that whenever $L$ is
  a linear operator on $\ell_2^n$ with $\|Le_i\|=1$ for the canonical basis vectors $\{e_i\}_{i\le n}$,
  one can find a subset $\sigma \subset [n]$ of cardinality
  \[ |\sigma| \ge cn/\|L\|_2^2\]
  for which 
  \begin{equation} \label{eqn:wellinv} \left\|\sum_{i\in\sigma} a_iLe_i\right\|^2 \ge d\sum_{i\in \sigma}|a_i|^2\end{equation}
  for all scalars $\{a_i\}_{i\in\sigma}$.
\end{theorem}
This theorem has had significant applications in the local theory of Banach spaces and in the
  study of convex bodies in high dimensions. 
It is also considered 
  a step towards the resolution of the famous Kadison-Singer conjecture, 
  which asks if there exists 
   a partition of $[n]$ into a constant number of subsets $\sigma_1,\ldots,\sigma_k$ for which (\ref{eqn:wellinv}) holds.
Recently, the theorem has attracted attention in
   numerical analysis due to its connection with 
  the column subset selection problem, which seeks to select a `representative' subset of columns from a given matrix.
In particular, Tropp \cite{tropp} has developed a randomized polynomial time algorithm which finds the subset $\sigma$ efficiently.

Bourgain and Tzafriri's proof of Theorem \ref{thm:bt} uses probabilistic and functional analytic
  techniques and is non-constructive. 
In the original paper the theorem was shown to hold for $c=d\sim \frac{1}{10^{72}}$. 
Later on \cite{bt91}, the same authors proved it for $c=c(\epsilon)=c'\epsilon^2$ and $d=(1+\epsilon)^{-1}$
  for every $0<\epsilon<1$, where $c'$ is a universal (tiny) constant.
They were interested in the case when $\epsilon$ is small; the quadratic dependence of $c(\epsilon)$ on $\epsilon$ was shown to be necessary in \cite{berman}. 
In another regime, modern methods can be used to obtain the constants
  $c=1/128$ and $d=1/8\sqrt{2\pi}$ \cite{casazza,tropp}.

In this note, we present a short proof that uses only basic linear algebra, achieves much better constants, and contains
  a deterministic $O(n^4)$ time algorithm for finding the set $\sigma$.
Our method of proof involves building $\sigma$ iteratively using a `barrier' potential function. Such a method was used by 
  Batson and the authors in \cite{bss} to construct linear size spectral sparsifiers of graphs.

Specifically, we prove the following 
  generalization of Theorem~\ref{thm:bt}, in which $\|\cdot\|_2$ refers to the spectral (i.e., operator) norm and $\|\cdot\|_F$ refers to the Frobenius (i.e., Hilbert-Schmidt) norm.
\begin{theorem}\label{mainthm} 
Suppose $v_1,\ldots v_m\in\R^n$, $\sum_i v_iv_i^T=I$, and $0<\epsilon
<1$. Let $L:\ell_2^n\to\ell_2^n$ be a linear operator.
Then there is a subset $\sigma\subset [m]$ of size $|\sigma|\ge
\left\lfloor \epsilon^2\frac{\|L\|_F^2}{\|L\|_2^2} \right\rfloor$ for
which $\{Lv_i\}_{i\in\sigma}$ is linearly independent and
\[\lambda_\mathrm{min}\left(\sum_{i\in \sigma} (Lv_i)(Lv_i)^T\right)> \frac{(1-\epsilon)^2
\|L\|_F^2}{m},\]
where $\lambda_\mathrm{min}$ is computed on $\mathrm{span}\{Lv_i\}_{i\in \sigma}$.
\end{theorem}
This form of generalization   was introduced by 
  Vershynin \cite{versh} in his study of contact points of convex bodies via 
  John's decompositions of the identity.
It says that given any such decomposition and any $L:\ell_2^n\to\ell_2^n$, there is a part of the 
  decomposition on which $L$ is well-invertible whose size is proportional to the stable rank $\frac{\|L\|_F^2}{\|L\|_2^2}$.

The original form of Bourgain and Tzafriri's theorem follows quickly from Theorem \ref{mainthm} with constants
  \[ c(\epsilon)=\epsilon^2\quad\textrm{and}\quad d(\epsilon)=(1-\epsilon)^2\]
  by taking $\{v_i\}$ from the standard basis $\{e_i\}_{i\le n}$ and assuming $\|Le_i\|=1$. 
This dominates previous bounds in all regimes, for $\epsilon$ small and large.

\section{Proof of the Theorem}
We will build the matrix $A=\sum_{i\in \sigma}(Lv_i)(Lv_i)^T$ by an
iterative process that adds one vector to $\sigma$ in each step. The
process will be guided by the potential function\footnote{
This potential function was inspired by Stieltjes transform,
  which appears in the analysis of the eigenvalues of random matrices.
However, we are unaware of a formal connection.
This potential function is also related to, but is not identical to,
  the logarithmic barrier function used in Interior Point Algorithms
  for Linear Programming.}
\begin{align*}
  \Phi_b(A) &= \sum_i (Lv_i)^T(A-bI)^{-1}(Lv_i) \\&=
  \tr{L^T(A-bI)^{-1}L}\qquad\textrm{since $\sum_i v_iv_i^T=I$,}
\end{align*}
where the {\em barrier} $b$ is a real number that varies from step to
step.

Initially $A=0$, the barrier is at $b=b_0>0$, and the potential is
\[ \Phi_{b_0}(0)=\tr{L^T(0-b_0I)^{-1}L}=-\tr{L^T
  L}/b_0=-\frac{\|L\|_F^2}{b_0}.\] Each step of the process involves
adding some rank-one matrix $ww^T$ to $A$ where $w\in \{Lv_i\}_{i\le
  m}$ (if $w=Lv_j$ then this corresponds to adding $j$ to $\sigma$)
and shifting the barrier towards zero by some fixed amount $\delta>0$,
without increasing the potential. Specifically, we want
\[ \Phi_{b-\delta}(A+ww^T)\le \Phi_b(A).\] We will maintain the
invariant that after $k$ vectors have been added, $A$ has exactly $k$
nonzero eigenvalues, all greater than $b$.  Keeping the potential
small (in fact, sufficiently negative) will ensure that there is a
suitable vector to add at each step.

In any step of the process, we are only interested in vectors $w$
which add a new nonzero eigenvalue that is greater than
$b'=b-\delta$. These are identified in the following lemma, where
the notation $A\succeq B$ means that $A-B$ is positive semidefinite.
\begin{lemma}
  Suppose $A\succeq 0$ has $k$ nonzero eigenvalues, all greater than
  $b' > 0$.  If $w\neq 0$ and
  \begin{equation}\label{eqn:denom} w^T(A-b'I)^{-1}w <
    -1\end{equation}
  then $A+ww^T$ has $k+1$ nonzero eigenvalues 
  greater than $b'$. \end{lemma}

\begin{proof}
  Let $\lambda_1 \geq \dotsb \geq \lambda_k$ be the nonzero
  eigenvalues of $A$, and let $\lambda_{1}' \geq \dotsb \geq
  \lambda_{k+1}'$ be the $k+1$ largest eigenvalues of $A+w w^{T}$.  As
  the latter matrix is obtained from $A$ by the addition of a rank one
  positive semi-definite matrix, their eigenvalues interlace \cite{bss}:
  \[
  \lambda_{1}' \geq \lambda_{1} \geq \lambda_{2}' \geq \dotsb \geq
  \lambda_{k} \geq \lambda_{k+1}'.
  \]
  Consider the quantity
  \[ \tr{ (A-b'I)^{-1}} =\sum_{i\le
    k}\frac{1}{\lambda_i-b'}+\sum_{i>k} \frac{1}{0-b'},\] where we
  have written the positive and negative terms in the sum separately.
  By the Sherman-Morisson formula,
  \begin{equation}\label{eqn:kplus1}
    \tr{(A+ww^T-b'I)^{-1}} - \tr{ (A-b'I)^{-1}} 
    = - \frac{w^T(A-b'I)^{-2}w}{1+w^T(A-b'I)^{-1}w}.
  \end{equation}
  Since $w^T(A-b'I)^{-1}w < -1$, the denominator in the right-hand
  term is negative.  The numerator is positive since $A-b'I$ is
  non-singular and $(A-b'I)^{-2}\succeq 0$.  So, the right-hand side
  of \eqref{eqn:kplus1} is positive.

  On the other hand, a direct evaluation of this difference yields
  \begin{align*}
    0 & < \tr{ (A+ww^T-b'I)^{-1}} - \tr{ (A-b'I)^{-1}}
    \\
    & = \frac{1}{\lambda_{k+1}' - b'} - \frac{1}{0-b'} +
    \sum_{i=1}^{k} \frac{1}{\lambda_{i}' - b'} - \sum_{i=1}^{k}
    \frac{1}{\lambda_{i} - b'}
    \\
    &\le \frac{1}{\lambda_{k+1}'-b'}+\frac{1}{b'}\qquad\textrm{since $\frac{1}{\lambda_i'-b'}-\frac{1}{\lambda_i-b'}\le 0$ for all $i$ by interlacing.}
  \end{align*}
  As $\lambda_{k+1}'\ge 0$, this is only possible if
  $\lambda_{k+1}'>b'$, as desired.
\end{proof}

The updated potential after one step, as the barrier moves from $b$ to
$b'=b-\delta$, can be calculated using the Sherman-Morisson formula:
\begin{align*}
  \Phi_{b'}(A+ww^T) &= \tr{L^T(A-b'I+ww^T)^{-1}L} \\&=
  \tr{L^T(A-b'I)^{-1}L} - \frac{\tr{L^T(A-b'I)^{-1}w
      w^T(A-b'I)^{-1}L}}{1+w^T(A-b'I)^{-1}w} \\&=
  \tr{L^T(A-b'I)^{-1}L} - \frac{w^T(A-b'I)^{-1}L
    L^T(A-b'I)^{-1}w}{1+w^T(A-b'I)^{-1}w} \\&= \Phi_{b'}(A) -
  \frac{w^T(A-b'I)^{-1}LL^T(A-b'I)^{-1}w}{1+w^T(A-b'I)^{-1}w}.
\end{align*}
To prevent an increase in potential, we want choose a $w$ such that
\begin{equation}
  \label{eqn:update}
  \Phi_{b'}(A)- \frac{w^T(A-b'I)^{-1}LL^T(A-b'I)^{-1}w}{1+w^T(A-b'I)^{-1}w}\le
  \Phi_{b}(A).
\end{equation}
We can now determine how small we need the potential to be in order to
guarantee that a suitable $w$, which will allow us to keep on going,
always exists.
\begin{lemma}\label{lem:avg}
  Suppose $A$ has $k$ nonzero eigenvalues, all of which
  are greater than $b$, and let $Q$ be the orthogonal projection onto
  the kernel of $A$. If
  \begin{equation}\label{eqn:reqpot}\Phi_b(A)\le
    -m-\frac{\|L\|_2^2}{\delta}\end{equation} 
  and 
  \begin{equation}\label{eqn:req2} 0<\delta<b\le\delta
    \frac{\|QL\|_F^2}{\|L\|_2^2}\end{equation} 
then
  there exists a vector $w\in \{Lv_i\}_{i\le m}$ for which $A+ww^T$ has $k+1$ nonzero eigenvalues greater than $b'=b-\delta$ and 
  $\Phi_{b'}(A+ww^T)\le \Phi_b(A)$.\end{lemma}
\begin{proof}\footnote{We would like to thank Pete Casazza for
    pointing out an important mistake in an earlier version of this
    proof.}
  The vectors satisfying {\em both} of the inequalities
   (\ref{eqn:denom}) and  (\ref{eqn:update}) are precisely those $w$ for
  which
  \begin{align*}
    & w^T (A-b'I)^{-1}LL^T(A-b'I)^{-1}w \\
    & \qquad \le (\Phi_{b}(A)-\Phi_{b'}(A))\cdot(-1-w^T(A-b'I)^{-1}w).
  \end{align*}
  We can show that such a $w$ exists by taking the sum over all
  $w\in\{Lv_i\}_{i\le m}$ and ensuring that the inequality holds in
  the sum, i.e., that
  \begin{align}\label{eqn:avg}
    &\tr{ L^T(A-b'I)^{-1}LL^T(A-b'I)^{-1}L}\notag\\
    &\qquad\le
    (\Phi_{b}(A)-\Phi_{b'}(A))\cdot(-m-\tr{L^T(A-b'I)^{-1}L)}).
  \end{align}
  Let $\dPhi := \Phi_b(A)-\Phi_{b'}(A)$.  From the assumption
  $\Phi_b(A)\le -m-\frac{\|L\|_2^2}{\delta}$ we immediately have
  \[ \tr{L^T(A-b'I)^{-1}L} = \Phi_b(A)-\dPhi \le
  -m-\frac{\|L\|_2^2}{\delta}-\dPhi\] and so (\ref{eqn:avg}) will follow from
  \begin{equation}\label{eqn:avg2}
    \tr{ L^T(A-b'I)^{-1}LL^T(A-b'I)^{-1}L}
    \le 
    \dPhi\cdot\left(\frac{\|L\|_2^2}{\delta}+\dPhi\right).
  \end{equation}
  Noting that $LL^T\preceq \|L\|_2^2 I$, we can bound the left hand
  side as
  \begin{equation}\label{eqn:l2} \tr{L^T(A-b'I)^{-1}L
      L^T(A-b'I)^{-1}L}
    \le 
    \|L\|_2^2 \tr{L^T(A-b'I)^{-2} L}.
  \end{equation}
  Let $P$ be the projection onto the image of $A$ and let $Q$ be the
  projection onto its kernel, so that $P+Q=I$.  Let
  $\Phi^P_{b'}(A)=\tr{L^TP(A-b'I)^{-1}PL}$ and
  $\Phi^Q_{b'}(A)=\tr{L^TQ(A-b'I)^{-1}QL}$ be the potentials computed
  on these subspaces.  Since $P$, $Q$, $A$, $(A-b'I)^{-1}$, and $(A-b'I)^{-2}$
  are mutually diagonalizable, we can write
  \[ \Phi_{b'}(A) = \Phi_{b'}^P(A)+\Phi_{b'}^Q(A),\qquad \dPhi =
  \dPhi^P+\dPhi^Q,\quad\textrm{ and}\]
  \[ \tr{L^T(A-b'I)^{-2}L} = \tr{L^TP(A-b'I)^{-2}PL} +
  \tr{L^TQ(A-b'I)^{-2}QL}.\] As $P(A-b'I)^{-1}P\succeq 0$ and
  $P(A-bI)^{-1}P\succeq 0$, it is easy to check that
  \[(b-b')P(A-b'I)^{-2}P\preceq P(A-bI)^{-1}P - P(A-b'I)^{-1}P\] which
  immediately gives
  \begin{equation}\label{eqn:dp} \|L\|_2^2\tr{L^TP(A-b'I)^{-2}PL} \le
    \dPhi^P\frac{\|L\|_2^2}{\delta}.\end{equation}
  Thus, by (\ref{eqn:avg2}), (\ref{eqn:l2}), and (\ref{eqn:dp}), we are done if we can show that
  \[
  \|L\|_2^2 \tr{L^TQ(A-b'I)^{-2}QL} \le
  (\dPhi^P+\dPhi^Q)\cdot\left(\frac{\|L\|_2^2}{\delta}+\dPhi\right)-\dPhi^P\frac{\|L\|_2^2}{\delta}.\]
  Taking into account that $\dPhi^P,\dPhi^Q\ge 0$, this is implied by the
  statement
  \begin{equation}\label{eqn:last}
    \|L\|_2^2 \tr{L^TQ(A-b'I)^{-2}QL} \le
    \dPhi^Q\cdot\left(\frac{\|L\|_2^2}{\delta}+\dPhi^Q\right).
  \end{equation}
  We now compute $ \tr{L^TQ(A-b'I)^{-2}QL} = \frac{\|QL\|_F^2}{b'^2} $
  and
  \[ \dPhi^Q = \tr{L^TQ ( ( A-bI)^{-1} - (A-b'I)^{-1} ) ) QL} =
  \delta\frac{\|QL\|_F^2}{bb'}\] which upon substituting and rearranging
  reduces (\ref{eqn:last}) to
  \[ \|L\|_2^2 \le \frac{\delta\|QL\|_F^2}{b}\] which we have assumed
  in (\ref{eqn:req2}).
\end{proof}

\begin{proof}[Proof of Theorem \ref{mainthm}]
We set
  \[
  b_0 = \frac{(1-\epsilon)\|L\|_F^2}{m} \qquad \text{and} \qquad
  \delta = \frac{(1-\epsilon)\|L\|_2^2}{\epsilon m}.
  \]
Requirement (\ref{eqn:reqpot}) of Lemma
  \ref{lem:avg} is satisfied at the beginning of the process 
  as
\[ \Phi_{b_0}(0)=-\frac{\|L\|_F^2}{b_0} =
  -m-\frac{\|L\|_2^2}{\delta}.
\]
To verify that requirement (\ref{eqn:req2}) is satisfied initially,
   first note that the theorem is vacuously true if 
   $\epsilon^{2} \frac{\|L\|_F^2}{\|L\|_2^2} < 1$.
Assuming the converse and recalling that $\epsilon < 1$, we may show
  $\frac{\|L\|_F^2}{\|L\|_2^2} \geq  1/\epsilon$
 which implies that $\delta < b_{0}$.
The inequality
  \[ b_0 \le\delta
  \frac{\|QL\|_F^2}{\|L\|_2^2}\]
  is initially true as $A = 0$ and so
    $Q=\mathrm{Proj}_{\mathrm{ker}(A)}=I$.

As long as condition (\ref{eqn:req2}) is satisfied, we may apply
  Lemma \ref{lem:avg} to add a vector to $\sigma$ while maintaining
  $\Phi_{b}(A) \leq  \Phi_{b_0}(0)$.
The left-hand inequality in (\ref{eqn:req2}) will be satisfied
  after the first $t-1$ steps if
\[
  \delta < b = b_{0} - (t-1) \delta \qquad \iff \qquad t \delta < b_{0}.
\]
This inequality is satisfied for all $t \leq \epsilon^{2} \frac{\|L\|_F^2}{\|L\|_2^2}$
as
\[
  \epsilon^{2} \frac{\|L\|_F^2}{\|L\|_2^2} \delta 
= 
  \frac{\epsilon (1-\epsilon) \|L\|_F^2}{m} < b_{0}.
\]
The right-hand inequality in (\ref{eqn:req2}) will always be satisfied
  if it is satisfied initially as 
  the Frobenius norm $\|QL\|_F^2$
  decreases by at most $\|L\|_2^2$ in each step.
Taking
  $t = \left\lfloor \epsilon^2\frac{\|L\|_F^2}{\|L\|_2^2} \right\rfloor$ steps leaves the barrier at
  \[
b_{0} - \delta t \geq  \frac{(1-\epsilon)\|L\|_F^2}{m} -
  \epsilon^2(1-\epsilon)\frac{\|L\|_F^2}{\epsilon m}
= 
\frac{(1-\epsilon)^{2} \|L\|_F^2 }{m}\] which is the
  promised bound.
\end{proof}
\section*{Acknowledgements}
We would like to thank Kate Juschenko, 
  Roman Vershynin and especially Pete Casazza for helpful comments and corrections 
  to an earlier version of this manuscript.

\end{document}